\documentclass{amsart}
\usepackage[T1]{fontenc}
\usepackage{textcomp}
\usepackage[utf8x]{inputenc}
\usepackage[english]{babel}
\usepackage{ucs}
\usepackage{graphicx}
\usepackage{mathrsfs}
\usepackage{rotating} 
\usepackage{tikz}
\usepackage{tikz-cd}
\usetikzlibrary{matrix}
\usetikzlibrary{calc,intersections}
\usepackage{rotating}
\usepackage{pgfplots}
\usepackage{epigraph}
\usetikzlibrary{arrows.meta,bending}
\usepackage{amssymb}
\usepackage[titletoc,toc]{appendix}
\usepackage{subcaption}
\usepackage{array}
\numberwithin{equation}{section} 
\newtheorem{thm}{Theorem}[section]

\newtheorem{cor}[thm]{Corollary}

\newtheorem{prop}[thm]{Proposition}

\theoremstyle{definition}
\newtheorem{definition}[thm]{Definition}
\theoremstyle{remark}
\newtheorem{rem}[thm]{Remark}
\numberwithin{equation}{section}
\newcommand{\bil}[1]{\left\langle #1 \right\rangle}

\newcommand{\Hi}[1]{\mathrm{H}^{#1}}

\newcommand{\Tr}{\mathrm{Tr} }

\newcommand{\derham}[2]{H^{#1} \! \left(#2\right)}

\newcommand{\PP}{\mathbb{P}}
\newcommand{\FF}{\mathbb{F}}

\newcommand{\QQ}{\mathbb{Q}}

\newcommand{\symp}[2]{\mathrm{Sp(#1,#2)}}
\newcommand{\Arf}[1]{\mathrm{Arf}\left(#1\right)}
\newcommand{\Jac}[1]{\mathrm{Jac}\left( #1 \right)}

\makeindex
\begin{document}

\title[Arithmetic and topology]{Arithmetic and topology of classical structures associated to plane quartics}%
\author{Olof Bergvall}%
\address{Department of Electronics, Mathematics and Natural Sciences, University of Gävle, Kungsbäcksvägen 47,
80176 Gävle, Sweden}
\email{olof.bergvall@hig.se}

% ----------------------------------------------------------------
\begin{abstract}
We consider moduli spaces of plane quartics marked with various structures such as
Cayley octads, Aronhold heptads, Steiner complexes and Göpel subsets and determine their
cohomology. This answers a series of questions of Jesse Wolfson.
We also explore some arithmetic applications over finite fields.
\end{abstract}

\maketitle
%\tableofcontents

% ----------------------------------------------------------------
\section{Introduction}
The systematic study of plane quartics goes back at least to the beginning
of the 19'th century.
The most famous result in this area is probably the fact that
every plane quartic has $28$ bitangents but there
are many structures other than bitangents which one can attach to 
a plane quartic.
Many of these structures have been well studied over the centuries
and include Cayley octads, Aronhold heptads, Steiner complexes and Göpel subsets.

For each of these structures, there is a corresponding moduli space
of plane quartics marked with this structure, e.g. the moduli space
$Q_{\mathrm{CO}}$ of plane quartics marked with a Cayley octad
and the moduli space $Q_{\mathrm{AH}}$ of plane quartics marked with an Aronhold heptad.
In \cite{owr}, Jesse Wolfson proposes the problem of computing the cohomology of
these spaces. The purpose of this note is to explain how these results
can be derived from our earlier works \cite{bergvalllic, bergvallthesis, bergvall_gd, bergvall_pts} 
(see also \cite{bergstrombergvall, bergvall_glp, bergvalltor, bergvallgounelas}). 
From a more modern perspective, these classical structures are naturally understood
in terms of level structures and subgroups of the symplectic group
$\symp{6}{\FF_2}$. In almost all cases (the case of Aronhold heptads is the exception), the corresponding subgroup
is maximal. Even though all of the structures associated to maximal subgroups of $\symp{6}{\FF_2}$
have been studied classically, two of these structures are considerably more
obscure; namely those associated to the maximal subgroups of $\symp{6}{\FF_2}$
of index $120$ respectively $960$. In fact, one of these structures does not even seem to have
been given a name before. We propose the name ``Riemann-Dickson system'' for
this structure since it seems to first have been discovered by Riemann (in a slightly different
setting) and brought into its geometric form by Dickson.

Plane quartics are closely related to both Del Pezzo surfaces of degree 2
and maximally nodal double Veronese cones, see \cite{ahmadinezhadetal},
and the above structures all have counterparts in these settings. 
Thus, most results in this note have counterparts in terms of Del Pezzo surfaces
and Veronese cones (e.g. refinements of results in \cite{banwaitetal} and \cite{loghrantrepalin}
in the sense of \cite{das}).
However, the precise nature of these structures and results are more subtle than
one might naively expect and we have therefore chosen to present these results elsewhere
in order not to obscure the explicit nature of the present paper.
It would also be very interesting to investigate the relationship between the approach
of the present paper with other recent developments in the theory of plane quartics 
and their bitangents; e.g. the tropical counts of Baker, Len, Morrison, Pflueger 
and Ren \cite{bakerlenmorrisonpfluegerren} and the signed counts of
Larson and Vogt \cite{larsonvogt}. 

We now outline the structure of the paper.
%In Section~\ref{sympsec} we recall some basic theory of symplectic vector
%spaces over the field of two elements and how this theory relates to
%curves with level two structure.
In Section~\ref{classsec} we recall the definitions of various structures related
to plane quartics and their bitangents, e.g. Cayley octads, Aronhold heptads and Steiner complexes. We also describe the subgroups of $\symp{6}{\FF_2}$
stabilizing these structures and give references to the classical literature,
especially in the more obscure cases (however, 
we make no claim to even closely exhaust the immense literature on
the topic).
In Section~\ref{cohsec} we compute the de Rham cohomology groups of the
moduli spaces of plane quartics marked with the various additional structures
considered in Section~\ref{classsec} and we count the number of isomorphism
classes of curves with such structures over finite fields of odd characteristic.
%Finally, in Section~\ref{othsec} we give some results with a similar flavour 
%but which did not naturally fit in elsewhere.

The paper contains more than a few moduli spaces. For convenience, we
provide a table listing the most important ones below.
We also list the most central results about these moduli spaces.

\begin{center}
{
\renewcommand{\arraystretch}{0.95}
\begin{tabular}{ll}
\multicolumn{2}{c}{\textbf{List of moduli spaces}} \\
\hline
% $\symp{n,k}$ & The symplectic group on a vector space of dimension $n$ over the field $k$, see Section~\ref{levelstructuressubsec}. \\
% $\weyl{\Phi}$ & The Weyl group of the root system $\Phi$, see Section~\ref{delpezzomarkedpointssubsec}. \\
$Q[2]$ & The moduli space of plane quartics with level $2$ structure  \\
$Q_{\mathrm{btg}}$ & The moduli space of plane quartics with a marked bitangent \\
$Q_{\mathrm{CO}}$ & The moduli space of plane quartics with a marked Cayley octad \\
$Q_{\mathrm{AH}}$ & The moduli space of plane quartics with a marked Aronhold heptad \\
$Q_{\mathrm{SC}}$ & The moduli space of plane quartics with a marked Steiner complex \\
$Q_{\mathrm{RD}}$ & The moduli space of plane quartics with a marked Riemann-Dickson system \\
$Q_{\mathrm{GS}}$ & The moduli space of plane quartics with a marked Göpel subset \\
$Q_{\mathrm{syz}}$ & The moduli space of plane quartics with a marked syzygetic tetrad \\
$Q_{\mathrm{azy}}$ & The moduli space of plane quartics with a marked Azygetic triad \\
$Q_{\mathrm{enn}}$ & The moduli space of plane quartics with a marked ennead
\end{tabular}
}
\end{center}

Recall that the Poincar\'e polynomial $P(X,t)$ of a space $X$ encodes the dimensions
of the cohomology groups of $X$. More precisely, the coefficient of $t^i$ is the
dimensions of $\mathrm{H}^i(X)$.

\begin{center}
{
\renewcommand{\arraystretch}{0.95}
\begin{tabular}{lcl}
\multicolumn{3}{c}{\textbf{The Poincar\'e polynomials of the various moduli spaces}} \\
\hline
% $\symp{n,k}$ & The symplectic group on a vector space of dimension $n$ over the field $k$, see Section~\ref{levelstructuressubsec}. \\
% $\weyl{\Phi}$ & The Weyl group of the root system $\Phi$, see Section~\ref{delpezzomarkedpointssubsec}. \\
$P(Q_{\mathrm{btg}},t)$ & = & $1+t^5+2t^6$ \\
$P(Q_{\mathrm{CO}},t)$ & = & $1+t+t^5+4t^6$ \\
$P(Q_{\mathrm{AH}},t)$ & = & $1+t+t^3+4t^4+6t^5+6t^6$ \\
$P(Q_{\mathrm{SC}},t)$ & = & $1+t+2t^5+5t^6$ \\
$P(Q_{\mathrm{RD}},t)$ & =& $1+2t^5+7t^6$ \\
$P(Q_{\mathrm{GS}},t)$ & =& $1+t+2t^5+11t^6$ \\
$P(Q_{\mathrm{syz}},t)$ & = & $1+t+t^4+7t^5+13t^6$ \\
$P(Q_{\mathrm{azy}},t)$ & = & $1+t+t^3+3t^4+8t^5+9t^6$ \\
$P(Q_{\mathrm{enn}},t)$ & = & $1+3t^3+11t^4+13t^5+11t^6$ 
\end{tabular}
}
\end{center}

\begin{center}
{
\renewcommand{\arraystretch}{0.95}
\begin{tabular}{lcl}
\multicolumn{3}{c}{\textbf{The number of points over $\FF_q$ of the various moduli spaces}} \\
\hline
% $\symp{n,k}$ & The symplectic group on a vector space of dimension $n$ over the field $k$, see Section~\ref{levelstructuressubsec}. \\
% $\weyl{\Phi}$ & The Weyl group of the root system $\Phi$, see Section~\ref{delpezzomarkedpointssubsec}. \\
$|Q_{\mathrm{btg}}(\FF_q)|$ & = & $q^6-q+2$ \\
$|Q_{\mathrm{CO}}(\FF_q)|$ & = & $q^6-q^5-q+4$ \\
$|Q_{\mathrm{AH}}(\FF_q)|$ & = & $q^6-q^5-q^3+4q^2-6q+6$ \\
$|Q_{\mathrm{SC}}(\FF_q)|$ & = & $q^6-q^5-2q+5$ \\
$|Q_{\mathrm{RD}}(\FF_q)|$ & =& $q^6-2q+7$ \\
$|Q_{\mathrm{GS}}(\FF_q)|$ & =& $q^6-q^5-2q+11$ \\
$|Q_{\mathrm{syz}}(\FF_q)|$ & = & $q^6-q^5+q^2-7q+13$ \\
$|Q_{\mathrm{azy}}(\FF_q)|$ & = & $q^6-q^5-q^3+3q^2-8q+9$ \\
$|Q_{\mathrm{enn}}(\FF_q)|$ & = & $q^6-3q^3+11q^2-13q+11$ 
\end{tabular}
}
\end{center}

\subsection*{Acknowledgements}
The author thanks Igor Dolgachev for providing some of the references to the classical literature
and Jesse Wolfson for comments and corrections.

\section{Classical constructions}
\label{classsec}
In this section we recall some classical results and constructions related to
plane quartics and their bitangents.

\subsection{Curves with symplectic level two structures}

Let $C$ be a smooth and irreducible projective 
curve of genus $g$ over an algebraically closed field of characteric different from $2$ 
and let $\Jac{C}$ be its Jacobian.
We will only consider group theoretic properties of
$\Jac{C}$ so we make the identifications
\begin{equation*}
 \Jac{C} = \mathrm{Pic}^0(C) = \mathrm{Cl}^0(C).
\end{equation*}
If $D \in \mathrm{Cl}(C)$, we denote the corresponding line bundle
by $\mathcal{L}(D)$ and we use the notation $h^n(D)$ for the dimension
of $\derham{n}{C,\mathcal{L}(D)}$.

Let $\Jac{C}[2]$  denote the $2$-torsion subgroup of $\Jac{C}$.
This group is evidently a vector space over $\FF_2$ and it is well known that its dimension is $2g$.
The Weil pairing $\bil{-,-}$ is a symplectic bilinear form on $\Jac{C}[2]$.

\begin{definition}
 A \emph{symplectic level two structure} on a curve $C$ of genus $g$
 is an isometry $\phi$ from the standard symplectic vector space of dimension $2g$
 to $(\Jac{C}\![2],b_C)$. 
\end{definition}
 
 Equivalently, a symplectic level two structure is
 a choice of an (ordered) symplectic basis $x_1, \ldots,x_g, y_1, \ldots, y_g$ of
 $(\Jac{C}\![2],b_C)$.  We will write $(C,\phi)$ and $(C,x_1, \ldots, x_g,y_1,\ldots,y_g)$
 interchangeably.
 
 Two curves with level two structures
 $(C,x_1, \ldots,y_g)$ and $(C',x'_1, \ldots, y'_g)$ are isomorphic
 if there is an isomorphism of curves $\varphi: C \to C'$ such that
 the induced morphism $\widetilde{\varphi}:\Jac{C'}\![2] \to \Jac{C}\![2]$
 takes one symplectic basis to the other in the sense that 
 \begin{align*}
  \widetilde{\varphi}(x'_i) & =x_i, \quad i=1, \ldots g, \\
  \widetilde{\varphi}(y'_i) & =y_i, \quad i=1, \ldots g.
 \end{align*}

\subsection{Bitangents and odd theta characteristics}

A \emph{plane quartic} is a smooth curve in the projective plane 
given by a degree $4$ polynomial. 
By the genus-degree formula, such a curve has genus $3$. Conversely, every non-hyperelliptic
curve of genus $3$ can be embedded as a plane quartic via its canonical linear system.
A \emph{bitangent} to a plane quartic $C$
is a line $L$ which is tangent to $C$ at two points or has contact order $4$ at one point.
If $L$ intersects $C$ at two distinct points we say that $L$ is a \emph{genuine bitangent}
and if $L$ only intersects $C$ at one point we say that it is a \emph{hyperflex line}.

There are many classical constructions and results centered around plane quartics and
their bitangents. The most famous is the following.

\begin{thm}[Pl\"ucker \cite{plucker}, Jacobi \cite{jacobi}]
 Every plane quartic has 28 bitangents.
\end{thm}

Let $K_C$ denote the canonical class of $C$. A \emph{theta characteristic} on a curve $C$
is a divisor class $\theta$ such that $2\theta = K_C$. 
A theta characteristic $\theta$ is called \emph{odd} (resp. \emph{even}) if the
\emph{Arf invariant} of $\theta$
\begin{equation*}
\mathrm{Arf}(\theta) = h^0(\theta) \mod 2
\end{equation*}
is $1$ (resp. $0$). A curve $C$ of genus $g$ has precisely $2^{2g}$ theta characteristics
of which $2^{g-1}(2^g-1)$ are odd and $2^{g-1}(2^g+1)$ are even. The set $\Theta$ of theta
characteristics of $C$ can be identified with the set $Q(\Jac{C}[2])$ of quadratic
forms on the symplectic vector space $\Jac{C}[2]$ via
\begin{equation*}
 \theta(v) = \Arf{\theta+v}+\Arf{\theta}.
\end{equation*}
We denote the subset of odd theta characteristics by $\Theta^-$ and the subset of even
theta characteristics by $\Theta^+$. 

Since a plane quartic $C$ is canonically embedded
into $\PP^2$, each effective canonical divisor of $C$ is given by intersecting $C$ with a line.
In particular, if $L$ is a bitangent then $L.C$ is a canonical divisor. For a bitangent
$L$ we have $L.C=2P+2Q$ for some points $P$ and $Q$ on $C$ (which may coincide if $L$ is a hyperflex
line) and we see that the class of $P+Q$ is a theta characteristic.
 This construction yields a bijection between the set of bitangents of $C$ and the set of
odd theta characteristics of $C$.

The stabilizer of a bitangent (or, if one prefers, an odd theta characteristic)
is a maximal subgroup of order $51840$. It is isomorphic to the Weyl group $W(E_6)$ of the root system $E_6$.
It has index $28$ in $\mathrm{Sp}(6,\FF_2)$.
The action of $W(E_6)$ is perhaps most easily understood in the world of Del Pezzo surfaces - 
here we can understand $W(E_6)$ as the group permuting the lines on a Del Pezzo surface $S$ of degree
$3$ (i.e. a cubic surface) and the Del Pezzo surface $S$ can in turn be thought of as being obtained by blowing
down a $(-1)$-curve on the Del Pezzo surface of degree $2$ obtained as the double cover of the plane quartic $C$
(see~\ref{aronholdsec} for a few more details on this perspective).

\subsection{Cayley octads and even theta characteristics}
\label{cayleyoctadssec}

There is also a close relationship between even theta characteristics of a non-hyperelliptic genus 
$3$ curve $C$ and bitangents of its canonical model as a plane quartic $Q$.
We explain this relationship below
(for a more complete account, see \cite{grossharris}). 

Let $\theta$ be an even theta characteristic of $C$ and consider
the linear system corresponding to the divisor class $K_C+\theta$. This linear
system is very ample and gives an embedding of $C$ into $\PP^3$ as a curve of degree $6$ which we
denote by $B$.
To be more precise, let $V=\Hi{0}(C,K_C)^*$, let $W=\Hi{0}(C,K_C+\theta)$
and consider the natural map $\varphi: V \to \mathrm{Sym}^2W^*$.
We make the identifications $\PP V = \PP^2$ and $\PP W = \PP^3$ and we think of $\varphi$ as
a net $\mathcal{N}$ (i.e. a linear system of dimension $2$) of quadrics in $\PP^3$.
With these identifications at hand, the quartic $Q$ is the locus 
\begin{equation*}
 Q := \{[v] \in \PP V| \varphi(v) \text{ is singular}\}
\end{equation*}
and the sextic $B \subset \PP W$ is the locus of singular points of members of $\mathcal{N}$. 
Moreover, for every point $P=[v]$ of $Q$ we may consider $\varphi(P)$ to be a map $W \to W^*$.
We let $L$ denote the divisor class corresponding to the dual of the line bundle whose fiber at $P \in Q$ is $\mathrm{ker}(\varphi(P))$. It can then be shown that
\begin{equation*}
 \theta = L-K_C.
\end{equation*}
Somewhat remarkably, the above process can be abstracted in the sense that starting
from two vector spaces $V$ and $W$ of dimensions $2$ and $3$ and a linear map 
$\varphi: V \to \mathrm{Sym}^2W^*$ one obtains a genus $3$ curve with an even
theta characteristic provided that $\varphi$ is sufficiently general,
i.e. provided that the net $\mathcal{N}$ given by $\varphi$ has eight points in general linear position as its
base locus. Not all $8$-tuples of points in $\PP^3$ occur as the base locus of a net
of quadrics but when it happens, the $8$-tuple uniquely determines the net.

\begin{definition}
 A \emph{Cayley octad} is an unordered $8$-tuple of points in $\PP^3$ 
 in general linear position which is the base locus of a net of quadrics.
\end{definition} 

Cayley octads are a special case of self-associated point sets,
see \cite{dolgachevortland}. Note that a plane quartic curve with an even theta characteristic uniquely defines a
Cayley octad and vice versa. In particular, we see that, up to projective equivalence, there are $36$ Cayley octads
associated to each plane quartic.

Let $\Omega=\{Q_1, \ldots, Q_8\}$ be a Cayley octad and let $B$ be the corresponding sextic curve
in $\PP^3$ given by the theta characteristic $\theta$. The $28$ lines $L_{ij}$, $1 \leq i < j \leq 8$ are bisecants to $B$, i.e. they intersect
$B$ in $2$ points each and each of the $28$ lines cuts out an odd theta characteristic on $B$.
If we now recall that the odd theta characteristics are naturally identified with bitangents
of the canonical model of $B$ as a plane quartic we see that $\theta$ can be specified
by labelling the $28$ bitangents with the $28$ pairs of points of a set of $8$ elements (in a compatible way according to the above
constructions). Sometimes Cayley octads are defined as such labellings, see e.g. \cite{conwayetal}.

The stabilizer of a Cayley octad is a maximal subgroup 
of order $40320$ and it is isomorphic to $S_8$, the symmetric
group on $8$ elements. It has index $36$ in $\mathrm{Sp}(6,\FF_2)$. 
Here, the action is plainly seen as the permutation action of $S_8$ on the eight points of 
the Cayley octad.

\subsection{Aronhold heptads}
\label{aronholdsec}
Let $\Omega=\{Q_1, \ldots, Q_8\}$ be a Cayley octad, let $B$ be the corresponding sextic curve
in $\PP^3$ and let $\theta$ be the corresponding even theta characteristic. 
By projecting from one of the points, say $Q_8$, we obtain seven points $P_1, \ldots, P_7$ in $\PP^2$ in general position,
i.e. no $3$ of them lie on a line and no $6$ of them lie on a conic. The image $\tilde{B}$ of $B$ under the projection is
a plane sextic curve with double points at the seven points. The $7$ lines $L_{18}, \ldots, L_{78}$ define
$7$ odd theta characteristics $\theta_1,\ldots, \theta_7$ on $B$. Such a $7$-tuple of odd theta characteristics
is called a \emph{Aronhold heptad}. A more direct definition is the following.

\begin{definition}
 An Aronhold heptad $\eta$ is a $7$-tuple of odd theta characteristics such that if
 $\theta_1, \theta_2$ and $\theta_3$ are distinct elements of $\eta$, then
 \begin{equation*}
  \theta_1+\theta_2+\theta_3
 \end{equation*} 
 is an even theta characteristic.
\end{definition}

Here, the sum should be interpreted as a sum of quadratic forms on $\symp{6}{\FF_2}$, not
as a sum of divisors (the corresponding divisor class is represented by e.g.
$\theta_1+\theta_2-\theta_3$ but the divisorial expressions are not symmetric in the indices).
We remark that we can get the even theta characteristic back from the Aronhold heptad via
\begin{equation*}
 \theta = \sum_{i=1}^7 \theta_i.
\end{equation*}
There are $8$ different points of the Cayley octad to
project from, thus yielding $8$ different Aronhold heptads. We thus see that there are $8 \cdot 36=288$
projectively inequivalent Aronhold heptads associated to a plane quartic.
We also take this opportunity to remark that if we blow up $P_1, \ldots, P_7$ we obtain
a Del Pezzo surface $S$ of degree $2$. We could also obtain $S$ more directly from $B$ by first taking the canonical model $Q$ of $B$ and then
taking the double cover branched along the quartic curve $Q \subset \PP^2$.
Thus, depending on our purposes we may choose to view a non-hyperelliptic genus $3$ curve
as a plane quartic, a space sextic, a plane sextic with $7$ double points or as the fixed
locus of the anticanonical involution of a Del Pezzo surface of degree $2$.

The stabilizer of a Aronhold heptad is a subgroup isomorphic to $S_7$, the symmetric
group on $7$ elements. This subgroup is not maximal (it is contained in the $S_8$ stabilizing the associated Cayley octad). The index of $S_7$ in $\mathrm{Sp}(6,\FF_2)$ is $288$. 
Again, the action is plainly seen as the permutation action of $S_7$ on the seven theta characteristics
 of the Aronhold heptad.

Both Cayley octads and Aronhold heptads on a plane quartic $C$ can be understood in terms of
symplectic level $2$ structures on $C$. To be precise, we have the following.

\begin{prop}
 An ordered Cayley octad uniquely determines a symplectic level $2$ structure and vice versa.
 An ordered Aronhold heptad uniquely determines a symplectic level $2$ structure and vice versa.
\end{prop}

For a proof, see \cite{dolgachevortland}. In fact, even more is true. If we let
$\mathcal{Q}[2]$ denote the moduli space of plane quartics with level $2$ structure,
let $\mathcal{C}_{\mathrm{ord}}$ denote the moduli space of ordered Cayley octads and
let $\mathcal{P}^2_{7,\mathrm{ord}}$ denote the moduli space of ordered $7$-tuples of
points in $\PP^2$ we have isomorphisms of moduli spaces
\begin{equation*}
 \mathcal{Q}[2] \cong \mathcal{C}_{\mathrm{ord}} \cong \mathcal{P}^2_{7,\mathrm{ord}}
\end{equation*}

\subsection{Steiner complexes}

Let $Q$ be a plane quartic and let $\mathcal{P}$ be the set of unordered distinct pairs
of odd theta characteristics of $Q$. Consider the map
\begin{equation*}
 s: \mathcal{P} \to \mathrm{Jac}(Q)[2]\setminus \{0\}
\end{equation*}
sending a pair $\{\theta_1, \theta_2\}$ to  $\theta_1+\theta_2$. 
The union of elements in a fiber of $s$ is called a
\emph{Steiner complex}, see \cite{dolgachev}. 

\begin{definition}
\label{steinerdef}
 Let $v \in \mathrm{Jac}(C)[2]$ be a nonzero element. The set 
 \begin{equation*}
  \Sigma(v) = \bigcup_{\{\theta_1,\theta_2\} \in s^{-1}(v)} \{\theta_1, \theta_2\}
 \end{equation*}
 is called the Steiner complex associated to $v$. 
\end{definition}

Thus, there is one Steiner complex for each of the $2^{2 \cdot 3}-1=63$ elements
of $\mathrm{Jac}(Q)[2]$. Each Steiner complex contains precisely $12$ odd theta characteristics
(or $12$ bitangents, if one prefers this viewpoint).

A more direct definition of a Steiner complex is
\begin{equation*}
 \Sigma(v) = \{\theta \in \Theta^-|\theta(v)=0\}. 
\end{equation*}
However, from Definition~\ref{steinerdef} it is clear how the $12$ elements of a Steiner complex
naturally form $6$ pairs. Permuting the 6 pairs stabilizes a Steiner complex
and so does interchanging the two elements of a pair. This suggests that the stabilizer subgroup
of a Steiner complex should be the wreath product $\FF_2 \wr S_6$ (i.e. the semidirect product 
$\FF_2^6 \rtimes S_6$ with the canonical action of $S_6$ on $\FF_2^6$). However, it turns out
(see e.g. \cite{dolgachev} for details) that a nontrivial parity condition must be satisfied
by the $\FF_2^6$-part; once five of the switches are chosen the sixth is determined.
This gives the stabilizer subgroup the structure of a semidirect product $\FF_2^5 \rtimes S_6$.
It is a maximal subgroup of cardinality $23040$ and index $63$. It is also possible to identify the stabilizer subgroup
with the Weyl group of $D_6$, see \cite{dolgachevortland}.

\subsection{Riemann-Dickson coordinates and octonial structures}
The material presented in this section, while classical, is not as well-known
as most of the other structures considered in this paper. For a more complete treatment and further
references, see Chapter XIX, Section 191 of \cite{milleretal} (for classical constructions) and 
\cite{manivel} (for connections to octonions).

Recall that the set $\Theta$ of theta characteristics is naturally a $\Jac{C}[2]$-torsor.
Riemann observed (and later also Weber, Clebsch, Appell and Goursat), in connection with his investigations of theta functions, that 
$\Theta$ is naturally parametrized by matrices of the form
\begin{equation*}
 \left( 
   \begin{array}{ccc}
    x_0 & x_1 & x_2  \\
    y_0 & y_1 & y_2
   \end{array}
 \right)
\end{equation*}
 with entries in $\FF_2$
such that odd (resp. even) theta characteristics correspond to matrices such that
\begin{equation*}
 x_0y_0+x_1y_1+x_2y_2 = 1 \quad (\text{resp. } 0).
\end{equation*}
Moreover, these coordinates can be chosen so that the four bitangents corresponding to
\begin{align*}
  \left( 
   \begin{array}{ccc}
    x_0 & x_1 & x_2  \\
    y_0 & y_1 & y_2
   \end{array}
 \right), & 
 \left( 
   \begin{array}{ccc}
    z_0 & z_1 & z_2  \\
    w_0 & w_1 & w_2
   \end{array}
 \right), \\
 \left( 
   \begin{array}{ccc}
    u_0 & u_1 & u_2  \\
    v_0 & v_1 & v_2
   \end{array}
 \right), & 
  \left( 
   \begin{array}{ccc}
    r_0 & r_1 & r_2  \\
    s_0 & s_1 & s_2
   \end{array}
 \right),
\end{align*}
intersect the curve $C$ in eight points lying on a conic if and only if
\begin{equation*}
 x_i+z_i+u_i+r_i=0, \quad y_i+w_i+v_i+s_i=0, \quad i=0,1,2,
\end{equation*}
(this geometric formulation seems to be due to Dickson, see \cite{milleretal}, p.373). 
We therefore propose the following terminology.

\begin{definition}
 We call a choice of coordinates for $\Theta$ satisfying the above conditions
 a choice of \emph{Riemann-Dickson coordinates}.
\end{definition}

This terminology is not standard but we have found no other in the literature.
There are $120$ possible choices of Riemann-Dickson coordinates for $\Theta$.

There is a quite remarkable connection between Riemann-Dickson coordinates and octonion multiplication
via the following simple construction. The Riemann-Dickson coordinates
\begin{equation*}
 \left( 
   \begin{array}{ccc}
    x_0 & x_1 & x_2  \\
    y_0 & y_1 & y_2
   \end{array}
 \right)
\end{equation*}
naturally correspond to two vectors $x=(x_0,x_1,x_3)$ and $y=(y_0,y_1,y_2)$ in the
vector space $\FF_2^3$. We label the unit octonions $e_0,e_1, \ldots, e_7$ with
elements of $\FF_2^3$ by identifying the vector $x=(x_0,x_1,x_2)$ with the
binary number $x_0x_1x_2$; e.g. $e_{(1,0,1)}=e_5$ and $e_{(0,0,0)}=e_0$.
Then, up to a sign, the formula
\begin{equation*}
 e_x \cdot e_y = e_{x+y}
\end{equation*}
is octonion multiplication. To recover the signs we identify the nonzero elements
of $\FF_2^3$ with the elements of the oriented Fano plane and use Freudenthal's
mnemonic \cite{freudenthal} for octonion multiplication; the elements of each line in the Fano plane are cyclically ordered, if the order of the multiplication is the same as the order of the line we obtain a positive sign, otherwise the sign is negative (and all elements except the neutral element square to minus the neutral element). We then obtain the usual Cayley-Graves table for octonion
multiplication, see Figure~\ref{fanofig} and Table~\ref{cayleygravestable}.

The stabilizer of a system of Riemann-Dickson coordinates is isomorphic to
the group $G_2(2)$ of automorphisms of the integral octonions. It is a maximal subgroup of cardinality $12096$
and index $120$.  

\begin{center}
\begin{figure}

\begin{tikzpicture}
\tikzstyle{point}=[ball color=white, circle, draw=black, inner sep=0.1cm]
\node (v4) at (0,0) [point] {$e_4$};
\draw (0,0) circle (1cm);
\draw[->] ({120*3-8}:1) arc ({120*3-8}:{120*2-8}:1);
\draw[->] ({120*2-8}:1) arc ({120*2-8}:{120-8}:1);
\draw[->] ({120-8}:1) arc ({120-8}:{-8}:1);
\node (v7) at (90:2cm) [point] {$e_7$};
\node (v6) at (210:2cm) [point] {$e_6$};
\node (v5) at (330:2cm) [point] {$e_5$};
\node (v1) at (150:1cm) [point] {$e_1$};
\node (v3) at (270:1cm) [point] {$e_3$};
\node (v2) at (30:1cm) [point] {$e_2$};
\draw[->] (v6) -- (v1);
\draw[->] (v1) -- (v7);
\draw[->] (v7) -- (v2);
\draw[->] (v2) -- (v5);
\draw[->] (v5) -- (v3);
\draw[->] (v3) -- (v6);
\draw[->] (v1) -- (v4);
\draw[->] (v4) -- (v5);
\draw[->] (v2) -- (v4);
\draw[->] (v4) -- (v6);
\draw[->] (v3) -- (v4);
\draw[->] (v4) -- (v7);
\end{tikzpicture}
\caption{The oriented Fano plane.}
\label{fanofig}
\end{figure}
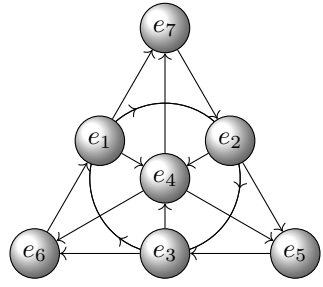
\end{center}

\begin{center}
\begin{table}
\begin{equation*}
\resizebox{\textwidth}{!}{$
 \begin{array}{r|rrrrrrrr}
  \, & e_{(0,0,0)} & e_{(0,0,1)} & e_{(0,1,0)} & e_{(0,1,1)} & e_{(1,0,0)} & e_{(1,0,1)} & e_{(1,1,0)} & e_{(1,1,1)} \\
  \hline
e_{(0,0,0)} & e_{(0,0,0)} & e_{(0,0,1)} & e_{(0,1,0)} & e_{(0,1,1)} & e_{(1,0,0)} & e_{(1,0,1)} & e_{(1,1,0)} & e_{(1,1,1)} \\ 
e_{(0,0,1)} & e_{(0,0,1)} & -e_{(0,0,0)} & e_{(0,1,1)} & -e_{(0,1,0)} &  e_{(1,0,1)} & - e_{(1,0,0)} & - e_{(1,1,1)} & e_{(1,1,0)} \\
e_{(0,1,0)} & e_{(0,1,0)} & - e_{(0,1,1)} & -e_{(0,0,0)} & e_{(0,0,1)} & e_{(1,1,0)} & e_{(1,1,1)} & -e_{(1,0,0)} & -e_{(1,0,1)}\\ 
e_{(0,1,1)} & e_{(0,1,1)} & e_{(0,1,0)} & -e_{(0,0,1)} & -e_{(0,0,0)} & e_{(1,1,1)} & -e_{(1,1,0)} & e_{(1,0,1)} & -e_{(1,0,0)}\\ 
e_{(1,0,0)} & e_{(1,0,0)} & -e_{(1,0,1)} & -e_{(1,1,0)} & -e_{(1,1,1)} & -e_{(0,0,0)} & e_{(0,0,1)} & e_{(0,1,0)} & e_{(0,1,1)}\\ 
e_{(1,0,1)} & e_{(1,0,1)} & e_{(1,0,0)} & -e_{(1,1,1)} & e_{(1,1,0)}  & -e_{(0,0,1)} & -e_{(0,0,0)} & -e_{(0,1,1)} & e_{(0,1,0)}\\
e_{(1,1,0)} & e_{(1,1,0)} & e_{(1,1,1)} & e_{(1,0,0)} & -e_{(1,0,1)} & -e_{(0,1,0)} & e_{(0,1,1)} & -e_{(0,0,0)} & -e_{(0,0,1)}\\
e_{(1,1,1)} & e_{(1,1,1)} & e_{(1,1,0)} & e_{(1,0,1)} & e_{(1,0,0)} & -e_{(0,1,1)} & -e_{(0,1,0)} & e_{(0,0,1)} & -e_{(0,0,0)}
 \end{array}$}
\end{equation*}
\caption{The octonion multiplication table.}
\label{cayleygravestable}
\end{table}
\end{center}

\subsection{Göpel subsets and maximal isotropic subspaces}

For more details and further perspectives, see \cite{coble}, \cite{dolgachevortland} and \cite{manivel}.

\begin{definition}
 Let $C$ be a plane quartic. A \emph{Göpel subspace} is a maximal isotropic
 subspace of $\Jac{C}[2]$ with respect to the Weil pairing. The set of the $7$ nonzero elements of a Göpel subspace
 is called a \emph{Göpel subset}.
\end{definition}

There are $135$ maximal isotropic subspaces of a symplectic vector space of dimension $6$ over $\FF_2$, see \cite{artin}.
Thus, there are $135$ Göpel subsets.

A Göpel subset is naturally a Fano plane. Therefore, the automorphism group $\mathrm{PGL}(3,\FF_2)$ 
of the Fano plane is naturally a subgroup of the stabilizer of a Göpel subset. To understand the rest of
the stabilizer it is convenient to once again recall the Del Pezzo picture. By taking the double cover of 
$\PP^2$ branched along a plane quartic $C$ we obtain a Del Pezzo surface $S$ of degree $2$. The surface
$S$ can also be obtained from blowing up (another) $\PP^2$ in seven points. An ordering of the seven points 
gives rise to seven ordered exceptional curves $D_1, \ldots, D_7$ in the Picard group $\mathrm{Pic}(S)$. The orthogonal complement
$K_S^{\perp}$ of $K_S$ in $\mathrm{Pic}(S)$ is naturally identified with the root lattice of the root system
$E_7$ and the curves $D_1, \ldots, D_7$ give rise to an ordered basis of $K_S^{\perp}$. Such a basis coming from a blow up is called a \emph{geometric marking}. A geometric marking of $S$ naturally corresponds to a level $2$ structure on $C$ (see, e.g., \cite{bergvallthesis} for details). Under this correspondence, the $63$ nonzero elements of $\Jac{C}[2]$ correspond to
the $63$ positive roots of the root system $E_7$ with respect to the geometric marking. In particular, the seven elements of a Göpel subset give rise to seven positive roots. The remainder of the stabilizer of a Göpel subset comes from the operation of switching such a root to its negative. However, we may only choose the direction of $6$ out of $7$ roots freely - once $6$ directions are chosen there is a unique direction of the final root so that the $7$ roots constitute part of a choice of positive roots for $E_7$. This explains how the stabilizer subgroup
of a Göpel subset is identified with $\mathrm{PGL}(3,\FF_2)\ltimes \FF_2^6$. It is a maximal subgroup of cardinality $10752$ and
index $135$.

\subsection{Syzygetic tetrads and isotropic planes}

\begin{definition}
 Let $\{\theta_1, \theta_2, \theta_3\}$ be a set of three odd theta characteristics on a curve
 $C$. The set is called an \emph{azygetic triad} if
 \begin{equation*}
  \mathrm{Arf}(\theta_1) +  \mathrm{Arf}(\theta_2) + \mathrm{Arf}(\theta_3) + \mathrm{Arf}(\theta_1+\theta_2+\theta_3) = 1,
 \end{equation*}
 otherwise it is called a \emph{syzygetic triad}.
\end{definition}

Note that if $\{\theta_1, \theta_2, \theta_3\}$ is a syzygetic triad of odd theta characteristics, 
then $\theta_{123}=\theta_1+\theta_2+\theta_3$ is an odd theta characteristic. Moreover,
any subset of three elements of $\{\theta_1, \theta_2,\theta_3, \theta_{123}\}$ is a syzygetic triad
such that the sum of the three elements is equal to the fourth.

\begin{definition}
 A \emph{syzygetic tetrad} is a set
 $\{\theta_1, \theta_2, \theta_3, \theta_4\}$ of four odd theta characteristics such that
 any subset of $3$ elements is a syzygetic triad and such that
 the sum of any three elements is equal to the fourth.
\end{definition}

In terms of bitangents, a syzygetic tetrad is a set of four bitangents such that
the intersection points of the bitangents and the quartic lie on a conic, see Figure~\ref{tetrad}.
Given a syzygetic tetrad of odd theta characteristics $\{\theta_1, \theta_2, \theta_3, \theta_4\}$
we may choose one of the four theta characteristics $\theta$ of the tetrad and construct
the plane
\begin{equation*}
 V = \{\theta_1-\theta, \theta_2-\theta, \theta_3-\theta, \theta_4-\theta \} \subset \Jac{C}[2].
\end{equation*}
One may show that $V$ is isotropic and independent of the choice of $\theta$, see Corollary 5.4.5
of \cite{dolgachev}. Furthermore, the isotropic planes of $\Jac{C}[2]$ correspond bijectively to the
syzygetic tetrads. Thus, there are $(2^6-1)\cdot (2^5-2)/|\mathrm{GL}(2,\FF_2)|=315$
syzygetic tetrads on a plane quartic.

 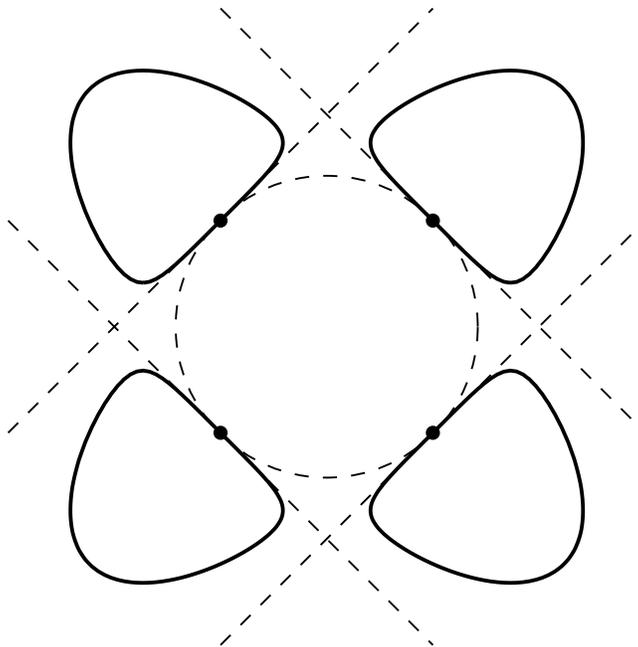
\begin{figure}[ht]
\centering
\resizebox{\linewidth}{!}{
 \begin{tikzpicture}
  \begin{axis}[axis lines=none,axis equal]
    \addplot +[id=edgeplot2,black,no markers, raw gnuplot, thick, empty line = jump % not strictly necessary, as this is the default behaviour in the development version of PGFPlots
      ] 
      gnuplot {
      set contour base;
      set cntrparam levels discrete 0.003;
      unset surface;
      set view map;
      set isosamples 500;
      set xrange [-3:3];
      set yrange [-3:3];
      splot x^4+y^4-6*(x^2+y^2)+10;
    };
    \addplot[dashed,black,domain=-3:1]{x+2};
    \addplot[dashed,black,domain=-1:3]{x-2};
    \addplot[dashed,black,domain=-1:3]{-x+2};
    \addplot[dashed,black,domain=-3:1]{-x-2};
    \draw[dashed,black] (300,300) circle [radius=32pt];
    \node [circle, black, fill, minimum size= 3pt, inner sep=0] at (400, 400) {};
    \node [circle, black, fill, minimum size= 3pt, inner sep=0] at (200, 400) {};
    \node [circle, black, fill, minimum size= 3pt, inner sep=0] at (200, 200) {};
    \node [circle, black, fill, minimum size= 3pt, inner sep=0] at (400, 200) {};
  \end{axis}
\end{tikzpicture}
}

\caption{A syzygetic tetrad on a plane quartic (the figure is reproduced from \cite{bergvallthesis}).}
\label{tetrad}
\end{figure}

The determination of the stabilizer of a syzygetic tetrad follows from standard
results around stabilizers of isotropic subspaces in symplectic spaces.
For completeness, we sketch the argument.

An element $\sigma$ of $\symp{6}{\FF_2}$ stabilizing the isotropic plane 
$V$ will necessarily stabilize the orthogonal complement $V^{\perp}$ of
$V$. Since $V$ is isotropic we have $V \subset V^{\perp}$. Thus,
$\sigma$ preserves the flag $0 \subset V \subset V^{\perp} \subset \Jac{C}[2]$.
The Weil pairing induces a symplectic pairing on the quotient $V^{\perp}/V$
so the stabilizer of $V$ contains a copy of $\mathrm{Sp}\left( V^{\perp}/V \right)\cong \symp{2}{\FF_2}$.
The stabilizer also contains a copy of $\mathrm{GL}(2, \FF_2)$ stabilizing $V$.
To identify the rest of the stabilizer we choose a symplectic basis $x_1,x_2,x_3, y_1, y_2, y_3$ such that
$V$ is spanned by $x_1$ and $x_2$ and $V^{\perp}$ is spanned by
$x_1,x_2, x_3$ and $y_3$; this is possible by Witt's lemma.
The elements of $\symp{6}{\FF_2}$ of the form
\begin{equation*}
\begin{array}{llcl}
A_{i,j}:& x_j & \mapsto & x_j + x_i \\
\, & y_i & \mapsto & y_i + y_j \\
B_{i,j}: & y_i & \mapsto & y_i + x_j \\
\, & y_j & \mapsto & y_j + x_i
\end{array}
\end{equation*} 
generate a non-abelian special $2$-group $G$ such that its center $Z(G)$ is an
abelian group isomorphic to $\FF_2^3$ and such that $G/Z(G)$ is isomorphic to $\FF_2^4$.
Thus, $G$ is an extension $G=\FF_2^3.\FF_2^4$.
These parts can be shown to constitute the full stabilizer; more precisely,
they fit together in a semidirect product $(\symp{2}{\FF_2} \times \mathrm{GL}(2,\FF_2)) \ltimes \FF_2^3.\FF_2^4$.

\begin{rem}
 Of course, $\symp{2}{\FF_2} \cong \mathrm{GL}(2,\FF_2) \cong S_3$ so one could in principle say that
 the stabilizer is $(S_3 \times S_3) \ltimes \FF_2^3.\FF_2^4$. This is the approach of
 \cite{conwayetal}. However, we found the above approach to be more transparent.
\end{rem}

\subsection{Azygetic triads of Steiner complexes}

A pair $\{\Sigma(u),\Sigma(v)\}$ of Steiner complexes is called \emph{syzygetic}
if $\bil{u,v}=0$ and it is called \emph{azygetic} if $\bil{u,v}=1$. 
A triple of three mutually syzygetic Steiner complexes is called a \emph{syzygetic triad
of Steiner complexes}. It can be shown that if $\{\Sigma(u),\Sigma(v),\Sigma(w)\}$ is a syzygetic
triad of Steiner complexes, then
\begin{equation*}
 \Sigma(u) \cup \Sigma(v) \cup \Sigma(w) = \Theta^-.
\end{equation*}
However, we are mainly interested in the opposite case since it corresponds
to a maximal subgroup of $\symp{6}{\FF_2}$.

\begin{definition}
 A triple $\{\Sigma(u),\Sigma(v),\Sigma(w)\}$ of Steiner complexes is called \emph{azygetic}
 if the vectors $u$, $v$ and $w$ form the nonzero vectors of a non-isotropic plane in $\Jac{C}[2]$.
\end{definition}

There are $336$ azygetic triads on a plane quartic. As mentioned above, there is a symmetric group
$S_6$ permuting pairs of elements of each Steiner complex (recall that a Steiner complex
naturally consists of $6$ pairs of odd theta characteristics) and there is a symmetric
group $S_3$ permuting the three Steiner complexes. This explains why the stabilizer subgroup
in $\symp{6}{\FF_2}$ of an azygetic tetrad of Steiner complexes can be identified with
the product $S_3 \times S_6$. It is a maximal subgroup of $\symp{6}{\FF_2}$ of cardinality $4320$ and
index $336$.

\subsection{Enneads and the Study quadric}

We have now covered structures preserved by all maximal subgroups of
$\symp{6}{\FF_2}$ except one - a subgroup of cardinality $1512$ and
index $960$. 
This subgroup was somewhat mysterious for some time but has now been studied
extensively by
Dye \cite{dye70}, Edge \cite{edge63} \cite{edge73}, Frame \cite{frame} and 
Study \cite{study} to mention a few. The material presented in this section is mainly due to them.

Cayley and Hesse denoted the $28$ bitangents by indexing them with pairs
of objects from a set of $8$ objects (Section~\ref{cayleyoctadssec} expands on this perspective).
Study observed that one may take the $8$ elements to be eight variables $x_1, \ldots, x_8$ and
the $28$ pairs to be the $28$ monomials $x_ix_j$, $1 \leq i <j \leq 8$. Then,
much of the geometry of the $28$ bitangents can be explored via the
\emph{Study quadric}
\begin{equation}
\label{studyeq}
 S= \, \sum_{1 \leq i <j \leq 8} x_ix_j, \quad x_i \in \FF_2. 
\end{equation}
For instance, $S$ defines a variety $V(S)$ in $\PP^7(\FF_2)$ with $135$ points
corresponding to the $135$ Göpel subsets and the $120$ points not on $V(S)$ correspond
to the $120$ systems of Riemann-Dickson coordinates. The lines in $\PP^7(\FF_2)$ fall
into different classes depending on the number of points they have in common with $V(S)$;
given a point $P$ outside $V(S)$ there are exactly $28$ lines through $P$ which do not
meet $V(S)$, $63$ lines through $P$ which meet $V(S)$ once and $36$ lines through $P$
which intersect $V(S)$ in two $\FF_2$-points. 

The form of $S$ given in Equation~\ref{studyeq} depends on the chosen coordinates for
$\PP^7(\FF_2)$ but there are many choices of coordinates for $\PP^7(\FF_2)$ which
preserve the form of $S$. To investigate the matter further, let $P_i$ denote
the point whose $i$th coordinate is $1$ and whose other coordinates are $0$ and
let $P_9$ denote the point whose coordinates are all $1$. 
The points $P_1, \ldots, P_9$ then all lie on $V(S)$
and the points $P_1, \ldots, P_8$
naturally correspond to the above choice of coordinates. Furthermore, any choice of $8$ points
among $P_1, \ldots, P_9$ corresponds to another choice of coordinates which leaves $S$ invariant.

Recall that the function $\Phi: \PP^7(\FF_2) \times \PP^7(\FF_2) \to \FF_2$
given by
\begin{equation*}
 \Phi(x,y) = \sum_{1 \leq i < j \leq 8} x_iy_j 
\end{equation*}
is called the \emph{polar form} with respect to $S$ and that two points $P$ and $Q$ in $\PP^7(\FF_2)$ such that
$\Phi(P,Q)=0$ are called \emph{conjugate} with respect to $S$.
We see that no two of the points $P_1, \ldots, P_9$ are conjugate with respect to $S$. 
Moreover, the chord joining any pair of the points $P_1, \ldots, P_9$ intersects
$V(S)$ precisely in those two points; the chord $L$ joining $P_i$ and $P_j$ contains $3$ points
and the remaining point has precisely $2$ nonzero coordinates - $V(S)$ does not
contain any points with precisely $2$ nonzero coordinates. It turns out that these two properties
characterize 9-tuples coming from choices of coordinates which gives $S$ the form of
Equation~\ref{studyeq}.

\begin{definition}
 A set of nine points on $V(S)$ such that
 \begin{itemize}
  \item no two points are conjugate with respect to $S$, and
  \item no chord between two points is contained in $V(S)$
 \end{itemize}
 is called an \emph{ennead}. 
\end{definition}

There are precisely $960$ enneads (and, thus, $960 \cdot 9!$ different choices of
coordinates preserving Equation~\ref{studyeq}).

Dye \cite{dye70} uses character theory to identify the stabilizer $G$ of an ennead
as a finite group of order $1512$ containing the group $\mathrm{PSL}(2,\FF_8)$ as
a maximal simple subgroup of index $3$. Up to isomorphism, there are exactly two
such groups. The character table of $G$ is given by Littlewood \cite{littlewood} p. 279.
The character table is not that of $\mathrm{PSL}(2,\FF_8) \times \FF_3$ so $G$ must be
the other possibility, namely the projective semilinear group $\mathrm{P}\Gamma\mathrm{L}(2,\FF_8)$ -
this group is also known as the Ree group $\mathrm{Ree}(3)$. The paper \cite{edge73} of Edge
is devoted to describing the action explicitly, we refer to his paper for details.
Dye has investigated $G$ in several works  after \cite{dye70}, see for instance
\cite{dye77, dye83b, dye83}. 

\subsection{Summary}

We summarize the results in Table~\ref{structuretable}.

\begin{center}
\begin{table}[h!]
\resizebox{\textwidth}{!}{
\begin{tabular}{l|c|r|r|l}
 Structure & Stabilizer & Size & \# & Max.?  \\
 \hline
 Bitangent (odd theta) & $W(E_6)$ & $51840$ & $28$ & Yes \\
 Cayley octad (even theta) & $S_8$ & $40320$ & $36$ & Yes \\
 Steiner complex & $\FF_2^5 \rtimes S_6$ (alt. $W(D_6)$) & $23040$ & $63$ & Yes \\
 Riemann-Dickson coordinates & $G_2(2)$ & $12096$ & $120$ & Yes \\
 Göpel subset (max. isotropic subspace) & $\mathrm{PGL}(3,\FF_2)\ltimes \FF_2^6$ & $10752$ & $135$ & Yes \\
 Aronhold heptad & $S_7$ & $5040$ & $288$ & No \\
 Syzygetic tetrad (isotropic plane) & $(\symp{2}{\FF_2} \times \mathrm{GL}(2,\FF_2) )\ltimes \FF_2^3.\FF_2^4$ & $4608$ & $315$ & Yes \\
 Azygetic triad & $S_3 \times S_6$ & $4320$ & $336$ & Yes \\
 Ennead & $\mathrm{P}\Gamma\mathrm{L}(2,\FF_8)$ (alt. $\mathrm{Ree}(3)$) & $1512$ & $960$ & Yes
\end{tabular}}
\caption{Various structures associated to a plane quartic curve, their stabilizers, the size of the stabilizer, the number of inequivalent structures and whether or not the stabilizer is a maximal subgroup of $\symp{6}{\FF_2}$.}
\label{structuretable}
\end{table}
\end{center}

\section{Cohomological computations}
\label{cohsec}

In this section we compute the cohomology of moduli spaces of plane quartic curves
marked with the various structures from Section~\ref{classsec}.

\subsection{Moduli of plane quartics with a marked bitangent (odd theta characteristic)}
The cohomology of the moduli space $Q_{\mathrm{btg}}$ of plane quartics with a marked bitangent
was first computed by Tommasi \cite{tommasi}.
Her computation used a Vassiliev type method.
Here we explain how to read off the result from our results in \cite{bergvall_gd} and
\cite{bergvall_pts} in two different ways.

\begin{thm}[Tommasi \cite{tommasi}]
\label{bitangentthm}
The dimensions of the rational de Rham cohomology groups of the moduli space of plane quartic curves with a marked bitangent line (i.e. an odd theta characteristic) are
\begin{equation*}
\begin{array}{lcl}
 \dim (\Hi{0}(Q_{\mathrm{btg}})) & = & 1 \\
 \dim (\Hi{1}(Q_{\mathrm{btg}})) & = & 0 \\
 \dim (\Hi{2}(Q_{\mathrm{btg}})) & = & 0 \\
 \dim (\Hi{3}(Q_{\mathrm{btg}})) & = & 0 \\
 \dim (\Hi{4}(Q_{\mathrm{btg}})) & = & 0 \\
 \dim (\Hi{5}(Q_{\mathrm{btg}})) & = & 1 \\
 \dim (\Hi{6}(Q_{\mathrm{btg}})) & = & 2
\end{array}
\end{equation*}
and $\dim(\Hi{i})=0$ for $i \neq 0, 5, 6$. The cohomology group $\Hi{i}$ is a pure Hodge structure
of type $(i,i)$. 
\end{thm}

\begin{proof}
 Version 1: The cohomology of the moduli space $Q_{\mathrm{btg}}[2]$ of plane quartics with a marked bitangent line and level two structure is given in Table~\ref{bitangenttable} as a representation of $\mathrm{Sp}(6,\FF_2)$.
 The cohomology of the quotient by $\mathrm{Sp}(6,\FF_2)$, i.e. the cohomology of $Q[2]$, can be read off
 as the invariant part, i.e. the part given by the trivial representation. This can be read off in column 1 of Table~\ref{bitangenttable}. The cohomology group $\Hi{i}(Q_{\mathrm{btg}}[2])$ is a pure Hodge structure of type $(i,i)$ by Lemma~2 and Equation~(1) of Section 6 and in \cite{bergvall_gd} so the same is true for the $\mathrm{Sp}(6,\FF_2)$-invariant part.
 
 Version 2: The cohomology of the moduli space $Q[2]$ of plane quartics with level two structure is given 
 in Table~\ref{Qtable} as a representation of $\mathrm{Sp}(6,\FF_2)$. 
 The stabilizer subgroup $G$ of a bitangent line  has character $\phi_{1a} + \phi_{27a}$ 
 (i.e. the representation $\mathrm{Ind}_{G}^{\mathrm{Sp}(6,\FF_2)} \mathrm{Triv}$ has character 
 $\phi_{1a} + \phi_{27a}$), see \cite{conwayetal}, p. 46. By Frobenius repricprocity, we obtain the dimension of $\Hi{i}(Q_{\mathrm{btg}})$
 by taking the inner product (in the sense of character theory) of $\phi_{1a} + \phi_{27a}$ and
 $\Hi{i}(Q_{\mathrm{btg}})$. Thus, we obtain $\mathrm{dim}(\Hi{i}(Q_{\mathrm{btg}}))$ by adding the multiplicities
 in the columns corresponding to $\phi_{1a}$ and $\phi_{27a}$ on the row corresponding to $\Hi{i}(Q[2])$. By Lemma 2, Section 6 och \cite{bergvall_gd} $\Hi{i}(Q[2])$ is a pure Hodge structure of
 type $(i,i)$ so the same is true for $\Hi{i}(Q_{\mathrm{btg}})=\Hi{i}(Q[2]/G)$. 
\end{proof}
 
\begin{cor}
\label{ptcountcor}
Let $q$ be a power of an odd prime number and let $\FF_q$ be a finite field with $q$ elements.
The number of pairs $(C,B)$ of a plane quartic curve $C$ and a bitangent line $B$ to $C$,
both defined over $\FF_q$, is 
$$q^6-q+2.$$ 
\end{cor}

\begin{proof}
 A standard argument using Artin's comparison theorem, constructibility, base
 change and Poincaré duality shows that Theorem~\ref{bitangentthm} can be interpreted in terms
 of compactly supported ètale cohomology over a field of odd characteristic.
 In the étale setting, the purity statement says that $Q_{\mathrm{btg}}$ is minimally pure
 in the sense of Dimca and Lehrer \cite{dimcalehrer}, i.e. that the Frobenius endomorphism
 acts as $q^{k-\mathrm{dim}(Q_{\mathrm{btg}})}$ on $\mathrm{H}^k_{\text{\'et},c}(Q_{\mathrm{btg}},\QQ_{\ell})$.
 By the Lefschetz trace formula we have
 \begin{equation*}
 \begin{array}{lcl}
   |Q_{\mathrm{btg}}(\FF_q)| & = & 
  \sum_{k \geq 0} (-1)^k\Tr \left(F , \mathrm{H}^k_{\text{\'et},c}(Q_{\mathrm{btg}},\QQ_{\ell}) \right) = \\
  & = & \sum_{k \geq 0} (-1)^k \mathrm{\dim} \left( \mathrm{H}^k_{\text{\'et},c}(Q_{\mathrm{btg}},\QQ_{\ell}) \right)q^{k- \mathrm{dim}\left(\mathrm{H}^k_{\text{\'et},c}(Q_{\mathrm{btg}},\QQ_{\ell})\right)} = \\
  & = & q^6-q+2.
  \end{array}
 \end{equation*}
\end{proof}

\subsection{Moduli of plane quartics with a marked Cayley octad (even theta characteristic)}

We denote the moduli space of plane quartics
with a marked Cayley octad by $Q_{\mathrm{CO}}$.

\begin{thm}
\label{cayleythm}
The dimensions of the rational de Rham cohomology groups of the moduli space of plane quartic curves with a marked Cayley octad (i.e. an even theta characteristic) are
\begin{equation*}
\begin{array}{lcl}
 \dim (\Hi{0}(Q_{\mathrm{CO}})) & = & 1 \\
 \dim (\Hi{1}(Q_{\mathrm{CO}})) & = & 1 \\
 \dim (\Hi{2}(Q_{\mathrm{CO}})) & = & 0 \\
 \dim (\Hi{3}(Q_{\mathrm{CO}})) & = & 0 \\
 \dim (\Hi{4}(Q_{\mathrm{CO}})) & = & 0 \\
 \dim (\Hi{5}(Q_{\mathrm{CO}})) & = & 1 \\
 \dim (\Hi{6}(Q_{\mathrm{CO}})) & = & 4
\end{array}
\end{equation*}
The cohomology group $\Hi{i}$ is a pure Hodge structure of type $(i,i)$. 
\end{thm}

\begin{proof}
The cohomology of the moduli space $Q[2]$ of plane quartics with level two structure is given 
 in Table~\ref{Qtable} as a representation of $\mathrm{Sp}(6,\FF_2)$. 
 The stabilizer subgroup $G$ of a Cayley octad  has character $\phi_{1a} + \phi_{35b}$ 
 (i.e. the representation $\mathrm{Ind}_{G}^{\mathrm{Sp}(6,\FF_2)} \mathrm{Triv}$ has character 
 $\phi_{1a} + \phi_{35b}$), see \cite{conwayetal}, p. 46. By Frobenius repricprocity, we obtain the dimension of $\Hi{i}(Q_{\mathrm{btg}})$
 by taking the inner product (in the sense of character theory) of $\phi_{1a} + \phi_{35b}$ and
 $\Hi{i}(Q_{\mathrm{btg}})$. Thus, we obtain $\mathrm{dim}(\Hi{i}(Q_{\mathrm{CO}}))$ by adding the multiplicities
 in the columns corresponding to $\phi_{1a}$ and $\phi_{35b}$ on the row corresponding to $\Hi{i}(Q[2])$. By Lemma 2, Section 6 och \cite{bergvall_gd} $\Hi{i}(Q[2])$ is a pure Hodge structure of
 type $(i,i)$ so the same is true for $\Hi{i}(Q_{\mathrm{btg}})=\Hi{i}(Q[2]/G)$. 
\end{proof}

A computation analogous to the proof of Corollary~\ref{ptcountcor} gives the following.

\begin{cor}
Let $q$ be a power of an odd prime number and let $\FF_q$ be a finite field with $q$ elements.
The number of pairs $(C,O)$ of a plane quartic curve $C$ and a Cayley octad $O$ of $C$,
both defined over $\FF_q$, is 
$$q^6-q^5-q+4.$$ 
\end{cor}

See also \cite{elsenhansjahnel} for some related computations and constructions.

\subsection{Moduli of plane quartics with a marked Aronhold heptad}

We denote the moduli space of plane quartics
with a marked Aronhold heptad by $Q_{\mathrm{AH}}$.

\begin{thm}
The dimensions of the rational de Rham cohomology groups of the moduli space of plane quartic curves with a marked Aronhold heptad are
\begin{equation*}
\begin{array}{lcl}
 \dim (\Hi{0}(Q_{\mathrm{AH}})) & = & 1 \\
 \dim (\Hi{1}(Q_{\mathrm{AH}})) & = & 1 \\
 \dim (\Hi{2}(Q_{\mathrm{AH}})) & = & 0 \\
 \dim (\Hi{3}(Q_{\mathrm{AH}})) & = & 1 \\
 \dim (\Hi{4}(Q_{\mathrm{AH}})) & = & 4 \\
 \dim (\Hi{5}(Q_{\mathrm{AH}})) & = & 6 \\
 \dim (\Hi{6}(Q_{\mathrm{AH}})) & = & 6
\end{array}
\end{equation*}
The cohomology group $\Hi{i}$ is a pure Hodge structure of type $(i,i)$. 
\end{thm}

\begin{proof}
 Version 1: The stabilizer $G$ of an Aronhold heptad is isomorphic to the symmetric group $S_7$. 
 The cohomology of $Q[2]$ was computed as a representation of $G$ in $\symp{6}{\FF_2}$
 in \cite{bergvalllic} (see also \cite{bergvallthesis} and \cite{bergvall_pts}). Below, we reproduce the
 table for convenience. We use the notation $s_{\lambda}$ for the irreducible representation
 of $S_7$ corresponding to the partition $\lambda$ of $7$. In particular, $s_7$ denotes the trivial
 representation of $S_7$. We obtain the result by reading off the invariant part, i.e. the column corresponding
 to $s_7$. 
\begin{table}[htbp]
\begin{equation*}
\begin{array}{r|rrrrrrrrrr} 
\, & s_{7} & s_{6,1} & s_{5,2} & s_{5,1^2} & s_{4,3} & s_{4,2,1} & s_{4,1^3} & s_{3^2,1} & s_{3,2^2} & s_{3,2,1^2} \\
\hline
H^0 & 1 & 0 & 0 & 0 & 0 & 0 & 0 & 0 & 0 & 0 \\ 
H^1 & 1 & 1 & 1 & 0 & 1 & 0 & 0 & 0 & 0 & 0 \\ 
H^2 & 0 & 3 & 4 & 4 & 3 & 5 & 1 & 3 & 1 & 1 \\
H^3 & 1 & 8 & 14 & 18 & 14 & 30 & 16 & 16 & 12 & 18 \\ 
H^4 & 4 & 20 & 44 & 47 & 44 & 99 & 56 & 56 & 54 & 83 \\ 
H^5 & 6 & 33 & 76 & 76 & 72 & 178 & 97 & 104 & 105 & 169 \\ 
H^6 & 6 & 23 & 51 & 54 & 54 & 127 & 74 & 76 & 77 & 126 \\
\hline
\, & s_{3,1^4} & s_{2^3,1} & s_{2^2,1^3} & s_{2,1^5} & s_{1^7} & \, & \, & \, & \, & \, \\
\hline
H^0 & 0 & 0 & 0 & 0 & 0 & \,&\,&\,&\,&\,\\ 
H^1 & 0 & 0 & 0 & 0 & 0 & \,&\,&\,&\,&\,\\ 
H^2 & 0 & 0 & 0 & 0 & 0 & \,&\,&\,&\,&\,\\
H^3 & 4 & 6 & 3 & 0 & 0 & \,&\,&\,&\,&\,\\ 
H^4 & 32 & 31 & 25 & 6 & 1 & \,&\,&\,&\,&\,\\ 
H^5 & 71 & 65 & 64 & 26 & 3 & \,&\,&\,&\,&\,\\ 
H^6 & 54 & 54 & 50 & 22 & 5 & \,&\,&\,&\,&\,
\end{array}
\end{equation*}
\caption{The cohomology of $Q[2]$ as a representation of $S_7$ (see \cite{bergvalllic},\cite{bergvallthesis} and \cite{bergvall_pts}).}
\label{S7Qcohtable}
\end{table}

Version 2: Since $S_7$ is not a maximal subgroup of $\symp{6}{\FF_2}$, 
we do not find the character corresponding to $S_7$ in \cite{conwayetal}. However, we can
compute the character corresponding to $S_7$ quite easily (but tediously) in a number of ways - 
most straightforward is perhaps to use an explicit embedding of $S_7$ into $\symp{6}{\FF_2}$,
see for instance p. 60 of \cite{bergvallthesis}. The result is the character
$\phi_{1a}+\phi_{27a}+\phi_{35b}+\phi_{105b}+\phi_{120a}$. We can now read off the
cohomology of $Q_{\mathrm{AH}}$ from Table~\ref{Qtable} as in the proof of Theorem~\ref{cayleythm}.
\end{proof}

A computation analogous to the proof of Corollary~\ref{ptcountcor} gives the following.

\begin{cor}
Let $q$ be a power of an odd prime number and let $\FF_q$ be a finite field with $q$ elements.
The number of pairs $(C,A)$ of a plane quartic curve $C$ and an Aronhold heptad  $A$,
both defined over $\FF_q$, is 
$$q^6-q^5-q^3+4q^2-6q+6.$$
\end{cor}

\subsection{Moduli of plane quartics with a marked Steiner complex}

We denote the moduli space of plane quartics
with a marked Steiner complex by $Q_{\mathrm{SC}}$.

\begin{thm}
The dimensions of the rational de Rham cohomology groups of the moduli space of plane quartic curves with a marked Steiner complex are
\begin{equation*}
\begin{array}{lcl}
 \dim (\Hi{0}(Q_{\mathrm{SC}})) & = & 1 \\
 \dim (\Hi{1}(Q_{\mathrm{SC}})) & = & 1 \\
 \dim (\Hi{2}(Q_{\mathrm{SC}})) & = & 0 \\
 \dim (\Hi{3}(Q_{\mathrm{SC}})) & = & 0 \\
 \dim (\Hi{4}(Q_{\mathrm{SC}})) & = & 0 \\
 \dim (\Hi{5}(Q_{\mathrm{SC}})) & = & 2 \\
 \dim (\Hi{6}(Q_{\mathrm{SC}})) & = & 5
\end{array}
\end{equation*}
The cohomology group $\Hi{i}$ is a pure Hodge structure of type $(i,i)$. 
\end{thm}

\begin{proof}
The proof follows the same path as the second version of the proof of Theorem~\ref{bitangentthm}
and the proof of Theorem~\ref{cayleythm}. Here, we must use that the 
stabilizer subgroup $G$ of a Steiner complex  has character $\phi_{1a} + \phi_{27a} + \phi_{35b}$ , see \cite{conwayetal}, p. 46. 
\end{proof}

A computation analogous to the proof of Corollary~\ref{ptcountcor} gives the following.

\begin{cor}
Let $q$ be a power of an odd prime number and let $\FF_q$ be a finite field with $q$ elements.
The number of pairs $(C,S)$ of a plane quartic curve $C$ and a Steiner complex $S$ of $C$,
both defined over $\FF_q$, is 
$$q^6-q^5-2q+5.$$ 
\end{cor}

\subsection{Moduli of plane quartics with a marked Riemann-Dickson system}
We denote the moduli space of plane quartics
with a chosen system of Riemann-Dickson coordinates by $Q_{\mathrm{RD}}$.

\begin{thm}
The dimensions of the rational de Rham cohomology groups of the moduli space of plane quartic curves with a marked Riemann-Dickson system are
\begin{equation*}
\begin{array}{lcl}
 \dim (\Hi{0}(Q_{\mathrm{RD}})) & = & 1 \\
 \dim (\Hi{1}(Q_{\mathrm{RD}})) & = & 0 \\
 \dim (\Hi{2}(Q_{\mathrm{RD}})) & = & 0 \\
 \dim (\Hi{3}(Q_{\mathrm{RD}})) & = & 0 \\
 \dim (\Hi{4}(Q_{\mathrm{RD}})) & = & 0 \\
 \dim (\Hi{5}(Q_{\mathrm{RD}})) & = & 2 \\
 \dim (\Hi{6}(Q_{\mathrm{RD}})) & = & 7
\end{array}
\end{equation*}
The cohomology group $\Hi{i}$ is a pure Hodge structure of type $(i,i)$. 
\end{thm}

\begin{proof}
The proof follows the same path as the second version of the proof of Theorem~\ref{bitangentthm}
and the proof of Theorem~\ref{cayleythm}. Here, we must use that the 
stabilizer subgroup $G$ of a Riemann-Dickson system  has character $\phi_{1a} + \phi_{35a} + \phi_{84a}$ , see \cite{conwayetal}, p. 46. 
\end{proof}

\begin{cor}
Let $q$ be a power of an odd prime number and let $\FF_q$ be a finite field with $q$ elements.
The number of pairs $(C,R)$ of a plane quartic curve $C$ and a Riemann-Dickson system $R$ of $C$,
both defined over $\FF_q$, is 
$$q^6-2q+7.$$ 
\end{cor}

\subsection{Moduli of plane quartics with a marked Göpel subset (maximal isotropic subspace)}

We denote the moduli space of plane quartics
with a marked Göpel subset (i.e. a maximal isotropic subspace) by $Q_{\mathrm{GS}}$.

\begin{thm}
The dimensions of the rational de Rham cohomology groups of the moduli space of plane quartic curves with a marked Göpel subset (i.e. a maximal isotropic subspace) are
\begin{equation*}
\begin{array}{lcl}
 \dim (\Hi{0}(Q_{\mathrm{GS}})) & = & 1 \\
 \dim (\Hi{1}(Q_{\mathrm{GS}})) & = & 1 \\
 \dim (\Hi{2}(Q_{\mathrm{GS}})) & = & 0 \\
 \dim (\Hi{3}(Q_{\mathrm{GS}})) & = & 0 \\
 \dim (\Hi{4}(Q_{\mathrm{GS}})) & = & 0 \\
 \dim (\Hi{5}(Q_{\mathrm{GS}})) & = & 2 \\
 \dim (\Hi{6}(Q_{\mathrm{GS}})) & = & 11
\end{array}
\end{equation*}
The cohomology group $\Hi{i}$ is a pure Hodge structure of type $(i,i)$. 
\end{thm}

\begin{proof}
The proof follows the same path as the second version of the proof of Theorem~\ref{bitangentthm}
and the proof of Theorem~\ref{cayleythm}. Here, we must use that the 
stabilizer subgroup $G$ of a Göpel subset has character $\phi_{1a} + \phi_{15a} +\phi_{35b} + \phi_{84a}$ , see \cite{conwayetal}, p. 46. 
\end{proof}

A computation analogous to the proof of Corollary~\ref{ptcountcor} gives the following.

\begin{cor}
Let $q$ be a power of an odd prime number and let $\FF_q$ be a finite field with $q$ elements.
The number of pairs $(C,G)$ of a plane quartic curve $C$ and a Göpel subset $G$ of $C$,
both defined over $\FF_q$, is 
$$q^6-q^5-2q+11.$$ 
\end{cor}

\subsection{Moduli of plane quartics with a marked syzygetic tetrad}

We denote the moduli space of plane quartics
with a marked syzygetic tetrad of bitangents by $Q_{\mathrm{syz}}$.

\begin{thm}
The dimensions of the rational de Rham cohomology groups of the moduli space of plane quartic curves with a marked syzygetic tetrad of bitangents are
\begin{equation*}
\begin{array}{lcl}
 \dim (\Hi{0}(Q_{\mathrm{syz}})) & = & 1 \\
 \dim (\Hi{1}(Q_{\mathrm{syz}})) & = & 1 \\
 \dim (\Hi{2}(Q_{\mathrm{syz}})) & = & 0 \\
 \dim (\Hi{3}(Q_{\mathrm{syz}})) & = & 0 \\
 \dim (\Hi{4}(Q_{\mathrm{syz}})) & = & 1 \\
 \dim (\Hi{5}(Q_{\mathrm{syz}})) & = & 7 \\
 \dim (\Hi{6}(Q_{\mathrm{syz}})) & = & 13
\end{array}
\end{equation*}
The cohomology group $\Hi{i}$ is a pure Hodge structure of type $(i,i)$. 
\end{thm}

\begin{proof}
The proof follows the same path as the second version of the proof of Theorem~\ref{bitangentthm}
and the proof of Theorem~\ref{cayleythm}. Here, we must use that the 
stabilizer subgroup $G$ of a syzygetic tetrad of bitangents has character $\phi_{1a} + \phi_{27a} +\phi_{35b} + \phi_{84a}+\phi_{168a}$ , see \cite{conwayetal}, p. 46. 
\end{proof}

A computation analogous to the proof of Corollary~\ref{ptcountcor} gives the following.

\begin{cor}
Let $q$ be a power of an odd prime number and let $\FF_q$ be a finite field with $q$ elements.
The number of pairs $(C,T)$ of a plane quartic curve $C$ and a syzygetic tetrad $T$ of $C$,
both defined over $\FF_q$, is 
$$q^6-q^5+q^2-7q+13.$$ 
\end{cor}

\subsection{Moduli of plane quartics with a marked azygetic triad}

We denote the moduli space of plane quartics
with a marked azygetic triad of Steiner complexes by $Q_{\mathrm{azy}}$.

\begin{thm}
The dimensions of the rational de Rham cohomology groups of the moduli space of plane quartic curves with a marked azygetic triad of Steiner complexes are
\begin{equation*}
\begin{array}{lcl}
 \dim (\Hi{0}(Q_{\mathrm{azy}})) & = & 1 \\
 \dim (\Hi{1}(Q_{\mathrm{azy}})) & = & 1 \\
 \dim (\Hi{2}(Q_{\mathrm{azy}})) & = & 0 \\
 \dim (\Hi{3}(Q_{\mathrm{azy}})) & = & 1 \\
 \dim (\Hi{4}(Q_{\mathrm{azy}})) & = & 3 \\
 \dim (\Hi{5}(Q_{\mathrm{azy}})) & = & 8 \\
 \dim (\Hi{6}(Q_{\mathrm{azy}})) & = & 9
\end{array}
\end{equation*}
The cohomology group $\Hi{i}$ is a pure Hodge structure of type $(i,i)$. 
\end{thm}

\begin{proof}
The proof follows the same path as the second version of the proof of Theorem~\ref{bitangentthm}
and the proof of Theorem~\ref{cayleythm}. Here, we must use that the 
stabilizer subgroup $G$ of an azygetic triad of Steiner complexes has character $\phi_{1a} + \phi_{27a} +\phi_{35b} + \phi_{105b}+\phi_{168a}$ , see \cite{conwayetal}, p. 46. 
\end{proof}

A computation analogous to the proof of Corollary~\ref{ptcountcor} gives the following.

\begin{cor}
Let $q$ be a power of an odd prime number and let $\FF_q$ be a finite field with $q$ elements.
The number of pairs $(C,T)$ of a plane quartic curve $C$ and an azygetic triad $T$ of $C$,
both defined over $\FF_q$, is 
$$q^6-q^5-q^3+3q^2-8q+9.$$. 
\end{cor}

\subsection{Moduli of plane quartics with a marked ennead}

We denote the moduli space of plane quartics
with a marked ennead by $Q_{\mathrm{enn}}$.

\begin{thm}
The dimensions of the rational de Rham cohomology groups of the moduli space of plane quartic curves with a marked ennead are
\begin{equation*}
\begin{array}{lcl}
 \dim (\Hi{0}(Q_{\mathrm{enn}})) & = & 1 \\
 \dim (\Hi{1}(Q_{\mathrm{enn}})) & = & 0 \\
 \dim (\Hi{2}(Q_{\mathrm{enn}})) & = & 0 \\
 \dim (\Hi{3}(Q_{\mathrm{enn}})) & = & 3 \\
 \dim (\Hi{4}(Q_{\mathrm{enn}})) & = & 11 \\
 \dim (\Hi{5}(Q_{\mathrm{enn}})) & = & 13 \\
 \dim (\Hi{6}(Q_{\mathrm{enn}})) & = & 11
\end{array}
\end{equation*}
The cohomology group $\Hi{i}$ is a pure Hodge structure of type $(i,i)$. 
\end{thm}

\begin{proof}
The proof follows the same path as the second version of the proof of Theorem~\ref{bitangentthm}
and the proof of Theorem~\ref{cayleythm}. Here, we must use that the 
stabilizer subgroup $G$ of an ennead has character 
$\phi_{1a} + \phi_{70a} +\phi_{84a} + \phi_{105b}+\phi_{280a}+\phi_{420a}$ , see \cite{conwayetal}, p. 46. 
\end{proof}

A computation analogous to the proof of Corollary~\ref{ptcountcor} gives the following.

\begin{cor}
Let $q$ be a power of an odd prime number and let $\FF_q$ be a finite field with $q$ elements.
The number of pairs $(C,E)$ of a plane quartic curve $C$ and an ennead $E$ of $C$,
both defined over $\FF_q$, is 
$$q^6-3q^3+11q^2-13q+11.$$ 
\end{cor}

 \begin{table}[htbp]
\begin{displaymath}
%\resizebox{\sz\textwidth}{!}{$
\resizebox{.45\vsize}{!}{$
\begin{array}{r|rrrrrrrrrr} 
\, & \phi_{1a} & \phi_{7a} & \phi_{15a} & \phi_{21a} & \phi_{21b} & \phi_{27a} & \phi_{35a} & \phi_{35b} & \phi_{56a} & \phi_{70a} \\
\hline
H^0 & 1 & \cdot & \cdot & \cdot & \cdot & \cdot & \cdot & \cdot & \cdot & \cdot \\
H^1 & \cdot & \cdot & \cdot & \cdot & \cdot & \cdot & \cdot & 1 & \cdot & \cdot \\
H^2 & \cdot & \cdot & \cdot & \cdot & \cdot & \cdot & \cdot & \cdot & \cdot & \cdot \\
H^3 & \cdot & \cdot & \cdot & 1 & \cdot & \cdot & \cdot & \cdot & \cdot & \cdot \\
H^4 & \cdot & \cdot & \cdot & \cdot & \cdot & \cdot & \cdot & \cdot & \cdot & 1 \\
H^5 & \cdot & \cdot & \cdot & \cdot & \cdot & 1 & 1 & 1 & \cdot & \cdot \\
H^6 & 1 & \cdot & 2 & \cdot & 1 & 1 & 1 & 3 & \cdot & \cdot \\
\hline
\, & \phi_{84a} & \phi_{105a} & \phi_{105b} & \phi_{105c} & \phi_{120a} & \phi_{168a} & \phi_{189a} & \phi_{189b} & \phi_{189c} & \phi_{210a} \\
\hline
H^0 & \cdot & \cdot & \cdot & \cdot & \cdot & \cdot & \cdot & \cdot & \cdot & \cdot \\
H^1 & \cdot & \cdot & \cdot & \cdot & \cdot & \cdot & \cdot & \cdot & \cdot & \cdot \\
H^2 & \cdot & \cdot & \cdot & \cdot & \cdot & \cdot & \cdot & \cdot & \cdot & 1 \\
H^3 & \cdot & \cdot & 1 & \cdot & \cdot & \cdot & 1 & \cdot & \cdot & 2 \\
H^4 & \cdot & \cdot & 2 & \cdot & 2 & 1 & 2 & 1 & \cdot & 3 \\
H^5 & 1 & 2 & 2 & 1 & 2 & 4 & 3 & 3 & 3 & 4 \\
H^6 & 5 & 1 & 1 & 4 & \cdot & 3 & 2 & 2 & 5 & 3 \\
\hline
\, & \phi_{210b} & \phi_{216a} & \phi_{280a} & \phi_{280b} & \phi_{315a} & \phi_{336a} & \phi_{378a} & \phi_{405a} & \phi_{420a} & \phi_{512a} \\
\hline
H^0 & \cdot & \cdot & \cdot & \cdot & \cdot & \cdot & \cdot & \cdot & \cdot & \cdot \\
H^1 & \cdot & \cdot & \cdot & \cdot & \cdot & \cdot & \cdot & \cdot & \cdot & \cdot \\
H^2 & \cdot & \cdot & \cdot & 1 & \cdot & \cdot & \cdot & \cdot & \cdot & \cdot \\
H^3 & 1 & \cdot & \cdot & \cdot & \cdot & \cdot & 1 & 2 & 2 & 1 \\
H^4 & 4 & \cdot & 3 & 1 & 3 & 2 & 3 & 6 & 5 & 4 \\
H^5 & 4 & 4 & 4 & 6 & 5 & 6 & 6 & 6 & 8 & 9 \\
H^6 & 1 & 6 & 3 & 6 & 1 & 6 & 4 & 2 & 6 & 6
\end{array}
$}
\end{displaymath}
\caption{The cohomology groups of the moduli space of plane quartics with level two structure as representations of $\mathrm{Sp}(6,\FF_2)$.}
\label{Qtable}
\end{table}

\begin{table}
\begin{displaymath}
%\resizebox{\sz\textwidth}{!}{$
\resizebox{.45\vsize}{!}{$
\begin{array}{r|rrrrrrrrrr} 
\, & \phi_{1a} & \phi_{7a} & \phi_{15a} & \phi_{21a} & \phi_{21b} & \phi_{27a} & \phi_{35a} & \phi_{35b} & \phi_{56a} & \phi_{70a} \\
\hline
H^0 & 1 & \cdot & \cdot & \cdot & \cdot & 1 & \cdot & \cdot & \cdot & \cdot \\
H^1 & \cdot & \cdot & \cdot & \cdot & \cdot & 1 & \cdot & 2 & \cdot & \cdot \\
H^2 & \cdot & \cdot & \cdot & 1 & \cdot & 1 & \cdot & 2 & \cdot & \cdot \\
H^3 & \cdot & \cdot & \cdot & 5 & \cdot & 3 & 1 & 3 & \cdot & 3 \\
H^4 & \cdot & \cdot & 1 & 7 & 1 & 8 & 7 & 9 & 8 & 17 \\
H^5 & 1 & 2 & 7 & 10 & 8 & 16 & 17 & 21 & 22 & 31 \\
H^6 & 2 & 2 & 10 & 6 & 10 & 13 & 14 & 20 & 16 & 20 \\
\hline
\, & \phi_{84a} & \phi_{105a} & \phi_{105b} & \phi_{105c} & \phi_{120a} & \phi_{168a} & \phi_{189a} & \phi_{189b} & \phi_{189c} & \phi_{210a} \\
\hline
H^0 & \cdot & \cdot & \cdot & \cdot & \cdot & \cdot & \cdot & \cdot & \cdot & \cdot \\
H^1 & \cdot & \cdot & 1 & \cdot & 1 & 1 & \cdot & \cdot & \cdot & 1 \\
H^2 & 1 & \cdot & 3 & 2 & 4 & 4 & 3 & \cdot & \cdot & 7   \\
H^3 & 5 & 1 & 11 & 8 & 15 & 16 & 19 & 3 & 5 & 25 \\
H^4 & 18 & 16 & 34 & 24 & 41 & 50 & 54 & 33 & 33 & 65 \\
H^5 & 43 & 46 & 54 & 50 & 62 & 89 & 92 & 83 & 86 & 106 \\
H^6 & 42 & 37 & 35 & 46 & 39 & 65 & 65 & 66 & 77 & 76 \\
\hline
\, & \phi_{210b} & \phi_{216a} & \phi_{280a} & \phi_{280b} & \phi_{315a} & \phi_{336a} & \phi_{378a} & \phi_{405a} & \phi_{420a} & \phi_{512a} \\
\hline
H^0 & \cdot & \cdot & \cdot & \cdot & \cdot & \cdot & \cdot & \cdot & \cdot & \cdot \\
H^1 & \cdot & \cdot & \cdot & 1 & \cdot & \cdot & \cdot & \cdot & \cdot & \cdot \\
H^2 & 2 & 1 & \cdot & 7 & \cdot & 2 & 1 & 6 & 4 & 4 \\
H^3 & 16 & 8 & 11 & 23 & 13 & 19 & 21 & 38 & 33 & 34 \\
H^4 & 60 & 42 & 64 & 74 & 73 & 79 & 89 & 122 & 114 & 130 \\
H^5 & 103 & 103 & 129 & 143 & 145 & 160 & 176 & 198 & 205 & 247 \\
H^6 & 68 & 90 & 95 & 112 & 100 & 126 & 131 & 129 & 151 & 181
\end{array}
$}
\end{displaymath}
\caption{The cohomology groups of the moduli space of plane quartics with a marked bitangent line and level two structure as representations of $\mathrm{Sp}(6,\FF_2)$.}
\label{bitangenttable}
\end{table}

% --------------------------
% \renewcommand\refname{Namn}
%\clearpage

\bibliographystyle{plain}

\renewcommand{\bibname}{References}

\bibliography{references} 
% ----------------------------------------------------------------
\end{document}